\documentclass{article}
\usepackage[utf8]{inputenc}
\usepackage{amsmath}
\usepackage{amsthm}
\usepackage{amsfonts} 

\title{On Numbers of Tuples of Nilpotent Matrices over Finite Fields under Simultaneous Conjugation}
\author{Jiuzhao Hua} 

\newtheorem{thm}{Theorem}[section]
\newtheorem{lem}{Lemma}[section]
\newtheorem{dfn}{Definition}[section]
\newtheorem{cjc}{Conjecture}[section]

\usepackage[parfill]{parskip}

\begin{document}
\date{\vspace{-0.5cm}}
\date{}  
\maketitle
\begin{abstract}
The problem of classifying tuples of nilpotent matrices over a field under simultaneous conjugation is considered ``hopeless". However, for any given matrix order over a finite field, the number of concerned orbits is always finite. This paper gives a closed formula for the number of absolutely indecomposable orbits using the same methodology as Hua \cite{JH 2000}; those orbits are non-splittable over field extensions. As a consequence, those numbers are always polynomials in the cardinality of the base field with integral coefficients. It is conjectured that those coefficients are always non-negative. 
\end{abstract}

\section{Introduction}
Let $q$ be a prime power and $\mathbb{F}_q$ be the finite field with $q$ elements. For any positive integer $n$, let $\mathcal{M}_n(\mathbb{F}_q)$ be the matrix algebra which consists of all $n\times n$ matrices over $\mathbb{F}_q$, $GL(n, \mathbb{F}_q)\subset\mathcal{M}_n(\mathbb{F}_q)$ be the General Linear Group consisting of all invertible ones and $\mathcal{N}_n(\mathbb{F}_q)\subset\mathcal{M}_n(\mathbb{F}_q) $ be the subset of all nilpotent ones.

Let $g$ be a fixed positive integer and $\mathcal{M}_n(\mathbb{F}_q)^g$ be the set of all $g$-tuples of $n\times n$ matrices over $\mathbb{F}_q$ and 
$\mathcal{N}_n(\mathbb{F}_q)^g$ be the set of all $g$-tuples of nilpotent ones, i.e.,
\begin{align*} 
\mathcal{M}_n(\mathbb{F}_q)^g &= \{(M_1, M_2, \dots, M_g) \,|\, M_i\in \mathcal{M}_n(\mathbb{F}_q), 1\le i \le g \}, \\
\mathcal{N}_n(\mathbb{F}_q)^g &= \{(M_1, M_2, \dots, M_g) \,|\, M_i\in \mathcal{N}_n(\mathbb{F}_q), 1\le i \le g \} .
\end{align*}
$GL(n, \mathbb{F}_q)$ acts on $\mathcal{M}_n(\mathbb{F}_q)^g$ by simultaneous conjugation, i.e.,
\begin{align*} 
GL(n, \mathbb{F}_q) \times \mathcal{M}_n(\mathbb{F}_q)^g \quad &\mapsto \quad \mathcal{M}_n( \mathbb{F}_q)^g \\
(T, (M_1, M_2, \dots, M_g) ) \quad &\mapsto \quad (T^{-1}M_1T, T^{-1}M_2T, \dots, T^{-1}M_gT).
\end{align*}
It is obvious that $\mathcal{N}_n(\mathbb{F}_q)^g$ is closed under the action of $GL(n, \mathbb{F}_q)$. Every $g$-tuple of matrices 
$(M_1, M_2, \dots, M_g)\in\mathcal{M}_n(\mathbb{F}_q)^g$ gives rise to a 
representation of the free algebra $\mathbb{F}_q\langle x_1, x_2, \dots, x_g\rangle$ by the following mapping:
\begin{align*} 
\mathbb{F}_q\langle x_1, x_2, \dots, x_g\rangle \quad &\mapsto \quad \mathcal{M}_n( \mathbb{F}_q) \\
x_i \quad &\mapsto \quad M_i \,\, (i=1,2,\cdots, g).
\end{align*}
Conversely, any finite dimensional representation of $\mathbb{F}_q\langle x_1, x_2, \dots, x_g\rangle$ is determined by a $g$-tuple from $\mathcal{M}_n(\mathbb{F}_q)^g$.
It is obvious that two $g$-tuples are in the same orbit if and only if their corresponding representations are isomorphic.

\begin{dfn}
An orbit of $\mathcal{M}_n(\mathbb{F}_q)^g/GL(n, \mathbb{F}_q)$ is said to be \textit{indecomposable} if its corresponding representation of the free algebra $\mathbb{F}_q\langle x_1, x_2, \dots, x_g\rangle$ is indecomposable; it is \textit{absolutely indecomposable} if its corresponding representation is absolutely indecomposable. 
\end{dfn}
Thus, an orbit $GL(n, \mathbb{F}_q) \cdot (M_1, M_2, \dots, M_g)$ is absolutely indecomposable if there does not exist an invertible matrix $T$ over $\overline{\mathbb{F}}_q$, the algebraic closure of $\mathbb{F}_q$, such that 
$$ (T^{-1}M_1T, T^{-1}M_2T, \dots, T^{-1}M_gT) = \left( \begin{bmatrix} A_1\!\!&\!\!0 \\ 0\!\!&\!\!B_1\end{bmatrix}, \begin{bmatrix} A_2\!\!&\!\!0 \\ 0\!\!&\!\!B_2\end{bmatrix},
\cdots, \begin{bmatrix} A_g\!\!&\!\!0 \\ 0\!\! &\!\!B_g\end{bmatrix} \right),$$
where $A_i, B_i$ for $1\le i\le g$ are square matrices over $\overline{\mathbb{F}}_q$.

Let $M_g(n,q)$ ($I_g(n,q)$, $A_g(n,q)$) be the number of orbits (indecomposable orbits, absolutely indecomposable orbits respectively) of $g$-tuples of nilpotent $n\times n$ matrices over $\mathbb{F}_q$ under simultaneous conjugation. It will be shown in later sections that $M_g(n,q), I_g(n,q)$ and $A_g(n,q)$ are all polynomials in $q$ with rational coefficients. The case for general $g$-tuples has been studied in Hua \cite{JH 2000}. It is widely known that $A_g(n,q)$'s are of significant importance because of their deep connections with Geometric Invariant Theory, Quantum Group Theory and Representation Theory of Kac-Moody Algebras (Kac \cite{VK 1983}, Ringel \cite{CR 1990} and Hausel \cite{TH 2010}). 

\section{Minimal building blocks of conjugacy classes of $GL(n,\mathbb{F}_q)$}

Let $\mathbb{N}$ be the set of all positive integers, $\mathcal{P}$ be the set of partitions of all positive integers, i.e.,
$$\mathcal{P} = \{(\lambda_1, \lambda_2, \dots, \lambda_k)\,|\,k\in\mathbb{N}, \lambda_i\in\mathbb{N}, \lambda_i\ge\lambda_{i+1}\ge 1, 1 \le i \le k\}.$$
The unique partition of $0$ is $(0)$. Let
$\Phi$ the set of monic irreducible polynomials in $\mathbb{F}_q[x]$ with $x$ excluded.
A \textit{partition valued function} on $\Phi$ is a function $\delta: \Phi \mapsto \mathcal{P}\cup\{(0)\}$. $\delta$ has \textit{finite support} if $\delta(f) = (0)$ except for finitely many $f$ in $\Phi$.

Let $f(x) = a_0 + a_1x + a_2x^2+ \dots + a_{n-1}x^{n-1} + x^n\in \mathbb{F}_q[x]$ and $c(f)$ be its \textit{companion matrix}, i.e.,
$$
\newcommand*{\temp}{\multicolumn{1}{|}{}}
c(f) = 
\left[
\begin{array}{ccccc}
0 & 1 & 0 & \dots & 0 \\ 
0 & 0 & 1 & \dots & 0 \\ 
\vdots & \vdots & \vdots & \ddots & \vdots \\
0 & 0 & 0 & \dots & 1 \\
-a_0 & -a_1 & -a_2 & \dots & -a_{n-1} 
\end{array}
\right].
$$
For any $m\in\mathbb{N}\backslash\{0\}$, let $J_m(f)$ be the \textit{Jordan block matrix} of order $m$ with $c(f)$ on the main diagonal, i.e.,
$$
\newcommand*{\temp}{\multicolumn{1}{|}{}}
J_m(f) = 
\left[
\begin{array}{ccccc}
c(f) & I & 0 & \dots & 0 \\ 
0 & c(f) & I & \dots & 0 \\ 
\vdots & \vdots & \vdots & \ddots & \vdots \\
0 & 0 & 0 & \dots & I \\
0 & 0 & 0 & \dots & c(f) 
\end{array}
\right]_{m\times m}, 
$$
where $I$ is the identity matrix of order $\deg(f)$.
For $\lambda=(\lambda_1, \lambda_2, \dots, \lambda_k)\in\mathcal{P}$, let $J_{\lambda}(f)$ be the 
\textit{direct sum} of $J_{\lambda_i}(f)$, i.e.,
$$J_{\lambda}(f) = J_{\lambda_1}(f) \oplus J_{\lambda_2}(f) \oplus \dots \oplus J_{\lambda_k}(f),$$
which stands for
$$
\newcommand*{\temp}{\multicolumn{1}{|}{}}
\left[
\begin{array}{cccc}
J_{\lambda_1}(f) & 0 & \dots & 0 \\ 
0 & J_{\lambda_2}(f) & \dots & 0 \\ 
\vdots & \vdots & \ddots & \vdots \\
0 & 0 & \dots & J_{\lambda_k}(f)
\end{array}
\right].
$$

\textbf{Rational Canonical Form Theorem} implies that, for any matrix $M\in GL(n,\mathbb{F}_q)$, there exists a unique partition valued function $\delta$ on $\Phi$ with finite support such that $\sum_{f\in\Phi}\text{deg}f \cdot |\delta(f)| = n$ and $M$ is conjugate to
$$\bigoplus_{f\in\Phi, \delta(f)\ne(0)}J_{\delta(f)}(f).$$
For this reason, $J_\lambda(f)$ where $\lambda\in\mathcal{P}$ and $f\in\Phi$ are called \textit{minimal building blocks} of conjugacy classes of $GL(n,\mathbb{F}_q)$. Rational Canonical Form for non-invertible matrices does exist as long as  $\Phi$ admits $x$ as its member.

\section{Nilpotent matrices commuting with minimal building blocks}

Any partition $\lambda\in\mathcal{P}$ can be written in its ``\textit{exponential form}" $(1^{n_1}2^{n_2}3^{n_3}\cdots)$, which means there are exactly $n_i$ parts in $\lambda$ equal to $i$ for all $i\ge 1$. The \textit{weight} of $\lambda$, denoted by $|\lambda|$, is $\sum_{i\ge 1}in_i$, and the \textit{length} of $\lambda$, denoted by $l(\lambda)$, is $\sum_{i\ge 1}n_i$. 
Let $\varphi_r(q)=(1-q)(1-q^2)\cdots(1-q^r)$ for $r\in\mathbb{N}$ and $\varphi_0(q)=1$. Furthermore, define $b_\lambda(q) = \prod_{i\ge1}\varphi_{n_i}(q)$.

\begin{dfn}
For any matrix of order $m\times n$, the \textit{arm length} of index $(i,j)$ is one plus the number of minimal moves from $(i,j)$ to $(1,n)$, where diagonal moves are not permitted. Thus the arm length distribution is as follows:
$$
\left[
\begin{array}{llllll}
n & n-1 & \dots & 3 & 2 & 1\\ 
n+1 & n & \dots & 4 & 3 & 2\\ 
n+2 & n+1 & \dots & 5 & 4 & 3\\ 
\vdots & \vdots & \vdots & \vdots & \vdots & \vdots\\ 
m+n & m+n-1 & \dots & m+2 & m+1 & m
\end{array}
\right]_{m\times n}.
$$
The \textit{arm rank} of a matrix $M = [a_{ij}]$ of order $m\times n$, denoted by $ar(M)$, is the largest arm length of indexes of non-zero elements of $M$, i.e.,
$$ar(M) = \max\left\{\textit{arm length of }(i,j) \,|\, a_{ij} \ne 0 \textit{ where } 1\le i\le m, 1\le j\le n\right\}.$$
\end{dfn}

\begin{dfn}
A matrix $M = [a_{ij}]$ of order $m\times n$ is of type-U if it satisfies the following conditions:
\begin{itemize}
\item $a_{ij} = a_{st}$ if $(i,j)$ and $(s,t)$ have the same arm length,
\item the arm rank of $M$ is at most $\min(m,n)$.
\end{itemize}
\end{dfn}
Thus a type-U matrix has either the following form when $m\ge n$:
$$
\newcommand*{\temp}{\multicolumn{1}{|}{}}
\left[
\begin{array}{lllll}
a_1 & a_2 & \dots & a_{n-1} & a_n \\ 
0 & a_1 & \dots & a_{n-2} & a_{n-1} \\ 
\vdots & \vdots & \ddots & \vdots & \vdots \\
0 & 0 & \dots & a_1 & a_2 \\
0 & 0 & \dots & 0 & a_1 \\
\cline{1-5}
0 & 0 & \dots & 0 & 0 \\
\vdots & \vdots & \vdots & \vdots & \vdots \\
0 & 0 & \dots & 0 & 0
\end{array}
\right]_{m \times n},
$$
or the following form when $m\le n$:
$$
\newcommand*{\temp}{\multicolumn{1}{|}{}}
\left[
\begin{array}{lllllllll}
0 & \dots & 0 & \temp & a_1 & a_2 & \dots & a_{m-1} & a_m \\ 
0 & \dots & 0 & \temp & 0 & a_1 & \dots & a_{m-2} & a_{m-1} \\ 
\vdots & \vdots & \vdots & \temp & \vdots & \vdots & \ddots & \vdots & \vdots \\
0 & \dots & 0 & \temp & 0 & 0 & \dots & a_1 & a_2 \\
0 & \dots & 0 & \temp & 0 & 0 & \dots & 0 & a_1 
\end{array}
\right]_{m \times n}. 
$$

\begin{thm}[Turnbull \& Aitken  \cite{T-A 1948}]\label{T-A Thm}
Let $\lambda=(\lambda_1, \lambda_2, \dots, \lambda_k)$ be a partition with $\lambda_ 1\ge \lambda_2 \ge\dots\ge\lambda_k\ge 1$ and $f(x)=x-a_0$ with $a_0\in\mathbb{F}_q$, then any matrix over $\mathbb{F}_q$ that commutes with $J_\lambda(f)$ can be written as a $k\times k$ block matrix in the following form:
$$
\newcommand*{\temp}{\multicolumn{1}{|}{}}
\left[
\begin{array}{cccc}
U_{11} & U_{12} & \dots & U_{1k} \\ 
U_{21} & U_{22} & \dots & U_{2k} \\ 
\vdots & \vdots & \ddots & \vdots \\
U_{k1} & U_{k2} & \dots & U_{kk} 
\end{array}
\right],
$$
where submatrix $U_{ij}$ is a type-U matrix over $\mathbb{F}_q$ of order $\lambda_i\times \lambda_j$ for all $(i,j)$ where $1\le i ,j\le k$.
\end{thm}

As an example, let $\lambda = (3, 2, 2)$ and $f(x) = x-t\in\mathbb{F}_q[x]$, $E$ a generic matrix that commutes with $J_\lambda(f)$, then 
$$\newcommand*{\temp}{\multicolumn{1}{|}{}}
J_\lambda(f)=\left[
\begin{array}{ccccccccc}
t & 1 & 0 & \temp & 0 & 0 & \temp & 0 & 0 \\ 
0 & t & 1 & \temp & 0 & 0 & \temp & 0 & 0 \\
0 & 0 & t & \temp & 0 & 0 & \temp & 0 & 0 \\
\cline{1-9}
0 & 0 & 0 & \temp & t & 1 & \temp & 0 & 0 \\ 
0 & 0 & 0 & \temp & 0 & t & \temp & 0 & 0 \\
\cline{1-9}
0 & 0 & 0 & \temp & 0 & 0 & \temp & t & 1 \\ 
0 & 0 & 0 & \temp & 0 & 0 & \temp & 0 & t 
\end{array}
\right]
, 
E=\left[
\begin{array}{cccccccccc}
a & b & c & \temp & l & m & \temp & p & q \\ 
0 & a & b & \temp & 0 & l & \temp & 0 & p \\
0 & 0 & a & \temp & 0 & 0 & \temp & 0 & 0 \\
\cline{1-9}
0 & r & s & \temp & d & e & \temp & h & i \\ 
0 & 0 & r & \temp & 0 & d & \temp & 0 & h \\
\cline{1-9}
0 & u & v & \temp & j & k & \temp & f & g \\ 
0 & 0 & u & \temp & 0 & j & \temp & 0 & f 
\end{array}
\right].
$$
\begin{thm}[Fine \& Herstein  \cite{F-H 1958}]\label{nil mat} For any positive integer $n$,
the number of nilpotent $n\times n$ matrices over $\mathbb{F}_q$ is equal to $q^{n^2-n}.$
\end{thm}

For a partition $\lambda=(\lambda_1, \lambda_2, \lambda_3, \dots)$, let $\lambda'=(\lambda_1', \lambda_2', \lambda_3', \dots)$ be its \textit{conjugate partition}, which means that $\lambda_i'$ is the number of parts in $\lambda$ that are greater than or equal to $i$ for all $i\ge 1$. 

\begin{dfn} 
Let $\lambda, \mu$ be two partitions and  $\lambda'=(\lambda_1', \lambda_2', \lambda_3', \dots)$ , $\mu'=(\mu_1', \mu_2', \mu_3', \dots)$ be their conjugate partitions. The``inner product" of $\lambda$ and $\mu$  is defined as follows:
$$\langle\lambda, \mu\rangle = \sum_{i\ge 1}\lambda_i'\mu_i'.$$
\end{dfn}

\begin{lem}[Hua \cite{JH 2000}]
Let $\lambda = (1^{m_1}2^{m_2}3^{m_3}\cdots)$ and $\mu = (1^{n_1}2^{n_2}3^{n_3}\cdots)$ be two partitions in their ``exponential form'', then there holds:
$$\langle\lambda, \mu\rangle = \sum_{i\ge 1}\sum_{j\ge 1} min (i, j)m_in_j.$$
\end{lem}

\begin{thm}\label{degree 1}
For any partition $\lambda\in\mathcal{P}$ and $f(x)=x-a_0\in\mathbb{F}_q[x]$, the number of nilpotent matrices over $\mathbb{F}_q$ that commute with $J_\lambda(f)$ is 
$q^{\langle\lambda, \lambda\rangle - l(\lambda)}.
$
\end{thm}
\begin{proof}
Suppose that $\lambda=(\lambda_1, \lambda_2, \dots, \lambda_k)$ with $\lambda_ 1\ge \lambda_2 \ge\dots\ge\lambda_k\ge 1$ and 
let $\mathcal{E}$ be the set of all matrices over $\mathbb{F}_q$ that commute with $J_\lambda(f)$, i.e.,
$$\mathcal{E} = \left\{M\in\mathcal{M}_v(\mathbb{F}_q)\, |\, MJ_\lambda(f) =J_\lambda(f)M, v = |\lambda|\right\}.$$
$\mathcal{E}$ is indeed the endomorphism algebra of the representation of $\mathbb{F}_q[x]$ induced by $J_\lambda(f)$. Theorem \ref{T-A Thm} implies that 
$$
\mathcal{E} = \left\{
\left[
\begin{array}{llll}
U_{11} & U_{12} & \dots & U_{1k} \\ 
U_{21} & U_{22} & \dots & U_{2k} \\ 
\vdots & \vdots & \ddots & \vdots \\
U_{k1} & U_{k2} & \dots & U_{kk} 
\end{array}
\right] \Bigg\rvert 
\begin{array}{l}
U_{ij} \textit{ is a type-U matrix over } \mathbb{F}_q,\\ 
U_{ij} \textit{ is of order } \lambda_i \times \lambda_j\textit{ for }1\le i,j\le k\\

\end{array}
\right\}.
$$

Let $(1^{n_1}2^{n_2}3^{n_3}\cdots)$ be the ``exponential form'' for $\lambda$. Thus there are $n_i$ parts equal to $i$. 
$\mathcal{E}$ has finite dimension over $\mathbb{F}_q$, 
the dimension that is contributed by the submatrices of order $i\times j$ for all paris $(i, j)$ is:
\begin{itemize}
\item $\min(i,j)n_in_j \textit{ if } i \ne j$,
\item $in_i^2 \textit{ if } i = j$.
\end{itemize}
Thus the dimension of $\mathcal{E}$ is:
$$\sum_{i\ge 1}\sum_{j\ge1, j\ne i} \min(i,j)n_in_j + \sum_{i\ge 1} in_i^2 = \sum_{i\ge 1}\sum_{j\ge 1} \min(i,j)n_in_j = \langle\lambda,\lambda\rangle.$$
It follows that the order of $\mathcal{E}$ is:
$$|\mathcal{E}| = q^{\langle\lambda,\lambda\rangle}.$$
Let $\mathcal{D}$ be the subspace of $\mathcal{E}$ defined as follows:
$$\mathcal{D} = \left\{
\left[
\begin{array}{llll}
D_{11} & D_{12} & \dots & D_{1k} \\ 
D_{21} & D_{22} & \dots & D_{2k} \\ 
\vdots & \vdots & \ddots & \vdots \\
D_{k1} & D_{k2} & \dots & D_{kk} 
\end{array}
\right]\in\mathcal{E}\,\Bigg| 
\begin{array}{l}
D_{ij} \textit{ is a matrix of order }\lambda_i \times \lambda_j, \\ 
D_{ij} = 0 \textit{ if } \lambda_i \ne \lambda_j, \\
D_{ij} = aI \textit{ for some $a\in\mathbb{F}_q$ if $\lambda_i=\lambda_j$}
\end{array}
\right\},
$$
and $\mathcal{N}$ be the subspace of $\mathcal{E}$ defined by:
$$
\mathcal{N} = \left\{
\left[
\begin{array}{llll}
N_{11} & N_{12} & \dots & N_{1k} \\ 
N_{21} & N_{22} & \dots & N_{2k} \\ 
\vdots & \vdots & \ddots & \vdots \\
N_{k1} & N_{k2} & \dots & N_{kk} 
\end{array}
\right]\in\mathcal{E}\,\Bigg| 
\begin{array}{l}
N_{ij} \textit{ is a type-U matrix of order } \lambda_i \times \lambda_j, \\ 
ar( N_{ij}) \le \lambda_i-1 \textit{ if } \lambda_i = \lambda_j.
\end{array}
\right\}.
$$
It is evident that $\mathcal{E}$ is a direct sum of $\mathcal{D}$ and $\mathcal{N}$ as a vector space. It can be verified that $\mathcal{N}$ is a two-sided ideal of $\mathcal{E}$. Furthermore $\mathcal{N}$ is nilpotent, i.e., every element in $\mathcal{N}$ is nilpotent. Thus every matrix $E\in\mathcal{E}$ can be uniquely written as a sum of a matrix from $\mathcal{D}$ and a matrix from $\mathcal{N}$, i.e.,
$$E = D + N \textit{ for some } D\in \mathcal{D} \textit{ and } N \in \mathcal{N}.$$
Since $\mathcal{N}$ is a nilpotent two-sided ideal of $\mathcal{E}$, $E$ is nilpotent if and only if $D$ is nilpotent.

As an example, if $\lambda = (3, 2, 2)$, then any matrix in $\mathcal{E}$ can be written as a sum as follows:
$$\newcommand*{\temp}{\multicolumn{1}{|}{}}
\left[
\begin{array}{ccccccccc}
a & 0 & 0 & \temp & 0 & 0 & \temp & 0 & 0 \\ 
0 & a & 0 & \temp & 0 & 0 & \temp & 0 & 0 \\
0 & 0 & a & \temp & 0 & 0 & \temp & 0 & 0 \\
\cline{1-9}
0 & 0 & 0 & \temp & d & 0 & \temp & h & 0 \\ 
0 & 0 & 0 & \temp & 0 & d & \temp & 0 & h \\
\cline{1-9}
0 & 0 & 0 & \temp & j & 0 & \temp & f & 0 \\ 
0 & 0 & 0 & \temp & 0 & j & \temp & 0 & f 
\end{array}
\right] +
\left[
\begin{array}{ccccccccc}
0 & b & c & \temp & l & m & \temp & p & q \\ 
0 & 0 & b & \temp & 0 & l & \temp & 0 & p \\
0 & 0 & 0 & \temp & 0 & 0 & \temp & 0 & 0 \\
\cline{1-9}
0 & r & s & \temp & 0 & e & \temp & 0 & i \\ 
0 & 0 & r & \temp & 0 & 0 & \temp & 0 & 0 \\
\cline{1-9}
0 & u & v & \temp & 0 & k & \temp & 0 & g \\ 
0 & 0 & u & \temp & 0 & 0 & \temp & 0 & 0 
\end{array}
\right].
$$
Every matrix $D\in\mathcal{D}$ can be viewed as a diagonal block matrix. Since there are $n_i$ parts equal to $i$ in $\lambda$, the block corresponding to $i^{n_i}$ is:
$$
D_i = 
\left[
\begin{array}{cccc}
a_{11}I & a_{12}I & \dots & a_{1n_i}I \\ 
a_{21}I &a_{22}I & \dots & a_{2n_i}I \\ 
\vdots & \vdots & \ddots & \vdots \\
a_{n_i1}I & a_{n_i2}I & \dots & a_{n_in_i}I 
\end{array}
\right],
$$
where $I$ is the identity matrix of order $i$ and all $a_{ij}\in\mathbb{F}_q$. Thus $D$ is nilpotent if and only if every diagonal block $D_i$ is nilpotent. Since $D_i$ is conjugate to the direct sum of $i$ copies of the following matrix:
$$
d_i = 
\left[
\begin{array}{cccc}
a_{11} & a_{12} & \dots & a_{1n_i} \\ 
a_{21} & a_{22} & \dots & a_{2n_i} \\ 
\vdots & \vdots & \ddots & \vdots \\
a_{n_i1} & a_{n_i2} & \dots & a_{n_in_i}
\end{array}
\right],
$$
$D_i$ is nilpotent if and only if $d_i$ is nilpotent. Since the number of nilpotent matrix of order $n_i$ over $\mathbb{F}_q$ is $q^{n_i^2-n_i}$ by Theorem \ref{nil mat}, the number of nilpotent 
matrices in $\mathcal{D}$ is:
$$q^{\sum_{i\ge 1}n_i^2-n_i}.$$
The dimension of $\mathcal{N}$ is:
$$\dim(\mathcal{E}) - \dim(\mathcal{D}) = \langle\lambda,\lambda\rangle - \sum_{i\ge 1}n_i^2.$$
Thus the order of $\mathcal{N}$ is:
$$|\mathcal{N}| = q^{ \langle\lambda,\lambda\rangle - \sum_{i\ge 1}n_i^2}.$$
Putting it all together, the number of nilpotent matrices in $\mathcal{E}$ is:
$$q^{\langle\lambda,\lambda\rangle - \sum_{i\ge 1}n_i^2}\cdot q^{\sum_{i\ge 1}n_i^2-n_i} 
= q^{\langle\lambda,\lambda\rangle - \sum_{i\ge 1}n_i} 
= q^{\langle\lambda,\lambda\rangle - l(\lambda)}.
$$
This finishes the proof.
\end{proof}

\begin{thm}\label{nil} For any partition $\lambda\in\mathcal{P}$ and any monic irreducible polynomial $f \in \mathbb{F}_q[x]$, the number of nilpotent matrices over $\mathbb{F}_q$ that commute with $J_\lambda(f)$ is $q^{d(\langle\lambda, \lambda\rangle - l(\lambda))}$,
where $d$ is the degree of $f$.
\end{thm}
\begin{proof}
Suppose that $d>1$ as the case for $d=1$ has been proved in Theorem \ref{degree 1}. Let $c(f)$ be the companion matrix for $f$ and $\langle c(f) \rangle$ be the subalgebra of $\mathcal{M}_d(\mathbb{F}_q)$ generated by $c(f)$. Since $f$ is the characteristic equation of $c(f)$, $c(f)$ satisfies the polynomial $f$, i.e., $f(c(f))=0$. Since $f$ is irreducible, $f$ is the minimal polynomial satisfied by $c(f)$. This implies that 
$I, c(f), c(f)^2, \cdots, c(f)^{d-1}$ form a basis for $\langle c(f) \rangle$ over $\mathbb{F}_q$, i.e.,
$$\langle c(f) \rangle = \left\{\sum_{i=0}^{d-1}a_ic(f)^i \,|\, a_i\in\mathbb{F}_q, 0\le i \le d-1 \right\}.$$
Thus $\langle c(f) \rangle$ is a commutative subalgebra of $\mathcal{M}_d(\mathbb{F}_q)$ and the following map is an isomorphism:
\begin{align*} 
\mathbb{F}_q[x]/(f(x)) &\mapsto \langle c(f) \rangle \\
x &\mapsto c(f).
\end{align*}
Since $f$ is irreducible, $\mathbb{F}_q[x]/(f(x))$ is isomorphic to the finite field $\mathbb{F}_{q^d}$, and hence $\langle c(f) \rangle$ is a finite field with $q^d$ elements. 

When $\deg(f) > 1$, Theorem \ref{T-A Thm} still holds as long as all submatrices $U_{ij}$ take values from the finite field $\langle c(f) \rangle$. All arguments in the proof of Theorem \ref{degree 1} still work with $\mathbb{F}_q$ being replaced by $\langle c(f) \rangle$. Thus Theorem \ref{degree 1} implies the desired results.
\end{proof}

\section{Calculating numbers of absolutely indecomposable orbits}
Let $\mathbb{Q}$ be the rational number field, $\mathbb{Q}[[X]]$ be the ring of formal power series in $X$ over $\mathbb{Q}$, $\mathbb{Q}(q)$ be the field of rational functions in $q$ over 
$\mathbb{Q}$ , and $\mathbb{Q}(q)[[X]]$ be the ring of formal power series in $X$ over $\mathbb{Q}(q)$.
Let $\phi_n(q)$ be the number of monic irreducible polynomials with degree $n$ in $\mathbb{F}_q[x]$ with $x$ excluded. It is known that for any positive integer $n$, 
$$\phi_n(q) = \frac{1}{n}\sum_{d\,|\,n}\mu(d)(q^{\frac{n}{d}}-1),$$
where the sum runs over all divisors of $n$ and $\mu$ is the Möbius function. Following Hua \cite{JH 2000}, let
$$P(X,q) = 1 + \sum_{\lambda\in\mathcal{P}} \frac{q^{g(\langle\lambda, \lambda\rangle - l(\lambda))}}{q^{\langle\lambda, \lambda\rangle}b_\lambda(q^{-1})}\,X^{|\lambda|}.$$

\begin{thm} The following identity holds in $\mathbb{Q}[[X]]$:
$$
1 + \sum_{n=1}^\infty M_g(n,q)X^n =\prod_{d=1}^\infty \left(P(X^d, q^d)\right) ^{\phi_d(q)}.
$$
\end{thm}
\begin{proof}
The method applied in Theorem 4.3 from Hua \cite{JH 2000} still works here.
In current context, the Burnside orbit counting formula is applied to $\mathcal{N}_n(\mathbb{F}_q)^g/GL(n, \mathbb{F}_q)$ and the number of points fixed by
$J_\lambda(f)$ is equal to $q^{\textup{deg}(f)g(\langle\lambda, \lambda\rangle - l(\lambda))}$ by Theorem \ref{nil}. Repeating the arguments there yields the desired result.
\end{proof}

\begin{dfn}
Define rational functions $H_g(n,q)$ for all positive integer $n$ as follows:
$$ 
\log\left(P(X, q)\right) = \sum_{n=1}^\infty H_g(n,q)X^n,
$$
where $\log$ is the formal logarithm, i.e., $\log(1+x) = \sum_{i\ge 1} (-1)^{i-1} x^i/i$.
\end{dfn}

\begin{thm}\label{A_g(n,q)}The following identity holds for all positive integer $n$:
$$
A_g(n,q) = (q-1) \sum_{d\,|\,n}\frac{\mu(d)}{d}H_g\Big(\frac{n}{d}, q^d\Big).
$$
\end{thm}
\begin{proof}
This is the counterpart of Theorem 4.6 from Hua \cite{JH 2000} with slight adjustment on the definition of $H_g(n,q)$, same arguments apply.
\end{proof}

Analogues of Theorem 4.6 of Hua \cite{JH 2000} have been proved by Bozec, Schiffmann \& Vasserot \cite{B-S-V 2018} for Lusztig nilpotent varieties and their variants 
using techniques from Algebraic Geometry. Their definition of nilpotency is stronger than the one used here. In the language of $\lambda$-ring and Adams operator, 
Theorem \ref{A_g(n,q)} is equivalent to the following identities in the ring of formal power series  $\mathbb{Q}(q)[[X]]$:
\begin{align*}
&\sum_{n=1}^\infty\!A_g(n,q)X^n = (q-1)\text{Log}\left(P(X,q)\right),  \\
&P(X, q) = \text{Exp}\left(\frac{1}{q-1} \sum_{n=1}^\infty\!A_g(n,q)X^n\right).
\end{align*}
For the definitions of operator $\text{Log}$ and $\text{Exp}$, we refer to the Appendix in Mozgovoy \cite{SM 2007}.

$H_g(n,q)$'s are rational functions in $q$, so are $A_g(n,q)$'s. As $A_g(n,q)$'s take integer values for all prime powers $q$, $A_g(n,q)$'s must be polynomials in $q$ with rational
coefficients. It follows from Lemma 2.9 of Bozec, Schiffmann \& Vasserot \cite{B-S-V 2018} that $A_g(n,q)\in\mathbb{Z}[q]$. Kac \cite{VK 1983} implies that the degree of 
polynomial $A_g(n,q)$ is at most $(g-1)n^2$. 
$I_g(n,q)$ and $M_g(n,q)$ can be calculated by the following identities:
\begin{align*}
&I_g(n,q) = \sum_{d\,|\,n}\frac{1}{d}\sum_{r\,|\,d}\mu\Big(\frac{d}{r}\Big)A_g\Big(\frac{n}{d},q^r\Big), \\
&1 + \sum_{n=1}^\infty M_g(n,q)X^n = \prod_{n=1}^\infty(1-X^n)^{-I_g(n,q)}.
\end{align*}
The first identity is the counterpart of the first identity of Theorem 4.1 from Hua \cite{JH 2000} and the second identity is a consequence of the Krull--Schmidt Theorem from representation theory. It follows that $I_g(n,q)$ and $M_g(n,q)$ are polynomials in $q$ with rational coefficients for all $n\ge 1$.

\begin{thm}\label{kwi}
Let
$$A_g(n,q) = \sum_{s=0}^{(g-1)n^2}a_{n, s} q^s,$$
where $a_{n, s}\in\mathbb{Z}$.
Then the following identity holds in $\mathbb{Q}(q)[[X]]$:
$$
P(X,q) = \prod_{n=1}^{\infty}\prod_{s=0}^{(g-1)n^2}\prod_{i=0}^{\infty} (1 - q^{s+i}X^n)^{a_{n,s}}.
$$
\end{thm}
\begin{proof}
This is the counterpart of Theorem 4.9 from Hua \cite{JH 2000}, same arguments apply.
\end{proof}

In the context of representations of quivers over finite fields, Kac \cite{VK 1983} conjectured that the constant term of the polynomial counting isomorphism classes of absolutely indecomposable representations with a given dimension vector is the same as the root multiplicity of the dimension vector in the corresponding Kac-Moody algebra. This conjecture was proved by Crawley-Boevey and Van den Bergh \cite{C-V 2004} for indivisible dimension vectors and 
by Hausel \cite{TH 2010} in general, which confirms that Theorem 4.9 from Hua \cite{JH 2000} is a $q$-deformation of Weyl-Kac denominator identity. Thus Theorem \ref{kwi} here may also be regarded as a $q$-deformation of Weyl-Kac denominator identity for some generalized Kac-Moody algebra.

When $g=1$, \textbf{Jordan Canonical Form Theorem} shows that there exists only one indecomposable nilpotent matrix over $\overline{\mathbb{F}}_q$ of a given order up to conjugation, where the unique conjugacy  class is the \textit{Jordan matrix} with eigenvalue $0$:
$$
\newcommand*{\temp}{\multicolumn{1}{|}{}}
\left[
\begin{array}{ccccc}
	0 & 1 & 0 & \dots & 0 \\ 
	0 & 0 & 1 & \dots & 0 \\ 
	\vdots & \vdots & \vdots & \ddots & \vdots \\
	0 & 0 & 0 & \dots & 1 \\
	0 & 0 & 0 & \dots & 0 
\end{array}
\right].
$$
This implies that $A_1(n,q)=1$ for all $n\ge 1$. Thus Theorem \ref{kwi} amounts to the following identity:
$$
1 + \sum_{\lambda\in\mathcal{P}} \frac{q^{- l(\lambda)}}{b_\lambda(q^{-1})}\,X^{|\lambda|} 
= \prod_{n=1}^{\infty}\prod_{i=0}^{\infty} (1 - q^{i}X^n).
$$

\begin{cjc}\label{cjc}
For any $g\ge 2$ and $n\ge 1$, all coefficients of the polynomial $A_g(n,q)$ are non-negative integers.
\end{cjc}

This conjecture is supported by the following observations generated by a Python program based on Theorem \ref{A_g(n,q)}:
\begin{align*}
A_2(1,q) = &\,\, 1, \\
A_2(2,q) = &\,\ 2q, \\
A_2(3,q) = &\,\, q^4 + 3q^2 + 2q, \\
A_2(4,q) = &\,\, q^9 + q^7 + q^6 + 4q^5 + 2q^4 + 7q^3 + 4q^2 +2q, \\
A_2(5,q) = &\,\, q^{16} + q^{14} + q^{13} + 2q^{12} + 2q^{11} + 4q^{10} + 4q^9 + 7q^8+ 8q^7 + 13q^6 +\\
&\,\, 13q^5 + 16q^4 + 14q^3 +7q^2 + 2q, \\
A_2(6,q) = &\,\,q^{25} + q^{23} + q^{22} + 2q^{21} + 2q^{20} + 4q^{19} + 3q^{18} + 7q^{17} + 7q^{16} + \\
&\,\, 10q^{15} + 11q^{14} + 19q^{13} + 17q^{12} + 28q^{11} + 29q^{10} + 39q^9 + 40q^8 + \\
&\,\, 53q^7 + 48q^6 + 52q^5 + 40q^4 + 25q^3 + 8q^2 + 2q.
\end{align*} 
\\
\vspace{0.2cm}
\textbf{\large{\!Acknowledgments}}
\vspace{0.2cm}

The author would like to thank Xueqing Chen, Bangming Deng, Jie Du and Yingbo Zhang for their helpful comments and suggestions.

\vspace{0.2cm}
Mathematics Enthusiast \\
\textit{Email address}: \texttt{jiuzhao.hua@gmail.com}

\end{document}